\documentclass[letterpaper]{article} 
\usepackage{authblk}
\title{Learning complexity of gradient descent and conjugate gradient algorithms}
\author{
Xianqi Jiao, Jia Liu, Zhiping Chen
}
\affil{    
     \footnotesize {School of Mathematics and Statistics, Xi'an Jiaotong University, Shaanxi, China}\\
    jxq00913@stu.xjtu.edu.cn, \{jialiu, zchen\}@xjtu.edu.cn

}

\usepackage{bibentry}

\usepackage{amsmath} 
\usepackage{amssymb}
\usepackage{array}
\usepackage{amsthm}
\usepackage{natbib}
\newtheorem{definition}{Definition}
\newtheorem{lemma}{Lemma}
\newtheorem{theorem}{Theorem}
\newtheorem{corollary}{Corollary}
\newtheorem{assumption}{Assumption}

\usepackage{comment}

\begin{document}

\maketitle

\begin{abstract}


Gradient Descent (GD) and Conjugate Gradient (CG) methods are among the most effective iterative algorithms for solving unconstrained optimization problems, particularly in machine learning and statistical modeling, where they are employed to minimize cost functions. In these algorithms, tunable parameters, such as step sizes or conjugate parameters, play a crucial role in determining key performance metrics, like runtime and solution quality. In this work, we introduce a framework that models algorithm selection as a statistical learning problem, and thus learning complexity can be estimated by the pseudo-dimension of the algorithm group. We first propose a new cost measure for unconstrained optimization algorithms, inspired by the concept of primal-dual integral in mixed-integer linear programming. Based on the new cost measure, we derive an improved upper bound for the pseudo-dimension of gradient descent algorithm group by discretizing the set of step size configurations. Moreover, we generalize our findings from gradient descent algorithm to the conjugate gradient algorithm group for the first time, and prove the existence a learning algorithm capable of probabilistically identifying the optimal algorithm with a sufficiently large sample size.
\end{abstract}

%

\section{Introduction}
Gradient descent (GD) and conjugate gradient (CG) methods are widely used, iterative algorithms for effectively solving unconstrained optimization problems, especially in machine learning and statistical modeling to minimize cost functions. Typically in GD \citep{s:05}, each iteration involves calculating the gradient at the current iteration point, multiplying it by a scaling factor (also known as step size), and subtracting it from the current iteration point. With certain assumptions on the problem's convexity, GD converges to a global minimum. Similarly, in CG \citep{h:52}, at each iteration step, the CG algorithm calculates the gradient at the current iteration point, determines the new search direction by combining the gradient of the current step with the previous direction multiplied by the conjugate parameter, and then adds the product of the search direction and the step size to the current iteration point. There are several well-known formulas for the conjugate parameter, including those proposed by 
\citet{h:52}, 
\citet{f:64}, 
\citet{p:69}, and 
\citet{d:99} . While different conjugate parameters yield varying convergence properties, it is well established that for an $n$-dimensional problem, the CG method theoretically converges in at most $n$ iterations for quadratic programs. In recent years, theoretical research on the CG method continues \cite{d:16,l:20,s:20,s:23,a:20}, but most studies focus exclusively on specific iterative methods and their convergence analyses.

Both GD and CG methods are widely utilized in various machine learning and optimization tasks, particularly in the training of models such as neural networks. Understanding the learning complexity of these algorithms is essential for evaluating their scalability in the face of increasing data sizes or model parameters. Furthermore, complexity analysis is closely related to the generalization properties of a learning algorithm, providing insights into the maximum number of samples required to effectively learn the optimal parameters (e.g., step size) within the algorithm. By estimating the error of a particular learning model alongside the learning complexity of the GD/CG methods employed in it, we can establish a robust theoretical guarantee regarding the model's performance. Optimization algorithms problems often have tunable parameters that significantly impact performance metrics such as runtime, solution quality, and memory usage, particularly in models incorporating machine learning techniques. {\it However, in most machine learning models, parameter tuning rarely comes with provable guarantees}. Therefore, {\it Selecting an appropriate step size or conjugate parameter is crucial for both GD and CG methods as they significantly affect convergence speed and generalization performance.} However, determining the appropriate step size and conjugate parameter for a specific type of problem remains an open question. Researchers have approached this challenge by treating the tuning of these parameters as a meta-optimization problem, where the goal is to identify the optimal parameters for a meta-optimizer. 
\citet{m:15} initially proposed a learning-to-learn (L2L) approach, followed by 
\citet{a:16}, who trained neural networks to predict the step size in the GD method. This learning-to-learn approach, also known as meta-learning \citep{l:16} or reversible learning \citep{m:15}, has gained considerable attention in recent years. Building on this framework, 
\citet{d:19} introduced a novel meta-algorithm designed to estimate the bias term online from a family of tasks, where the single task is to apply stochastic gradient descent (SGD) on the empirical risk, regularized by the squared Euclidean distance from a bias vector. 
\citet{h:23} further extended meta-learning by providing guarantees for various well-known game classes, including two-player zero-sum, general-sum, and Stackelberg games, and also proposed results for single-game scenarios within a unified framework. For a comprehensive overview of meta-learning models, we refer readers to a recent survey \citep{h:22}.

Recently, the machine learning community has shown a growing interest in theoretical guarantees. 
\citet{w:21} highlighted that meta-optimization is particularly challenging due to issues like exploding or vanishing meta-gradients and poor generalization if the meta-objective is not carefully selected. They provided meta-optimization guarantees for tuning the step size in quadratic loss functions, demonstrating that a naive objective can suffer from meta-gradient issues and that a carefully designed meta-objective can maintain polynomially bounded meta-gradients. 
\citet{y:22} introduced a stable adaptive stochastic conjugate gradient (SCG) algorithm tailored for the mini-batch setting. This algorithm integrates the stochastic recursive gradient approach with second-order information within a CG-type framework, allowing for the estimation of the step size sequence without requiring Hessian information. This approach achieves the low computational cost typical of first-order algorithms. Additionally, they established a linear convergence rate for SCG algorithms on strongly convex loss functions.

We would like to highlight a learning-theoretic framework for data-driven algorithm design introduced by 
\citet{g:16}, which established sample complexity bounds for tuning the step size in gradient descent. It
comprises four key components: a fixed optimization problem, an unknown instance distribution, a family of algorithms, and a cost function that evaluates the performance of a given algorithm on a specific instance. Building on Gupta’s work, data-driven algorithms have been applied across a wide range of optimization problems. 
\citet{b:17} utilized this framework for configuration problems in combinatorial partitioning and clustering. 
They further extended this approach to other optimization challenges, including tree search algorithms such as the branch-and-bound method \citep{b:18} and the branch-and-cut method \citep{b:22}, as well as for non-convex piecewise-Lipschitz functions \citep{b:21}. Furthermore, 
\citet{a:19} developed an algorithm for pruning problems that learns to optimally prune the search space in repeated computations.

In the vein of research on sample complexity for algorithms, {\it the choice of cost function is equally critical in assessing an algorithm’s performance}. Common options include running time and the objective function value of the solution produced by the algorithm \citep{g:16,b:2021}. In \citet{g:16}, for GD problems, the cost function is defined as the number of iterations (i.e., steps) the algorithm takes to solve a given optimization task. For clustering methods  \citep{b:17}, the cost function is set as the $k$-means or $k$-median objective value, the distance to a ground-truth clustering, or the expected value of the solution. In the branch-and-bound method \citep{b:22}, the cost function is set as the size of the search tree under a given node selection policy. In a word, selecting an appropriate cost function can significantly enhance performance and may also yield stronger theoretical guarantees.

To the best of our knowledge, when learning the parameter in GD, the cost function is typically defined as {\it the number of iterations}. For instance, 
\citet{g:16} show the existence of a learning algorithm that $(1+\epsilon,\delta)$-learns the optimal algorithm in an algorithm set $\mathcal{A}$ using $m=\tilde{O}(H^3/\epsilon^2)$ samples, here $H$ is a given constant. {\it However, extending this result to the CG method meets significant challenges.} Moreover, when the optimization problem is too large or computational resources are limited, calculating the iteration number based cost function additionally becomes impractical.  Furthermore, since the number of iterations must be an integer, the value $1+\epsilon$ cannot be further reduced. 
{\it To address these issues, we introduce a new cost function that sums the distances between the current and optimal values at each iteration.} We demonstrate that, using this cost function, there exists a learning algorithm that $(C+\epsilon,\delta)$-learns the optimal algorithm in the GD algorithm set $\mathcal{A}$ using $m=\tilde{O}(H^3/\epsilon^2)$ samples for any $C > 0$. We further extend the framework to the CG algorithm, and establish that 
there exists a learning algorithm that $(C+\epsilon,\delta)$-learns the optimal algorithm in the CG algorithm set $\mathcal{A}$ using $m=\tilde{O}(H^4/\epsilon^2)$  samples for any $C > 0$.

To conclude, the major contributions of this paper includes:
\begin{enumerate}

\item  We propose a new cost function to more effectively measure the performance of GD and CG algorithms. In detail, the function calculates the sum of distances between the current and optimal values at each iteration.

\item We improve the learning complexity of GD under the new cost function. In detail, using the new cost function, we demonstrate that there exists a learning algorithm that $(C+\epsilon,\delta)$-learns the optimal algorithm within algorithm set $\mathcal{A}$ using $m=\tilde{O}(H^3/\epsilon^2)$ samples for any $C > 0$.

\item To the best of our knowledge, we firstly 
establish the learning complexity result for the CG algorithm.
In detail, we prove that there exists a learning algorithm that $(C+\epsilon,\delta)$-learns the optimal algorithm within the CG algorithm set $\mathcal{A}$ using $m=\tilde{O}(H^4/\epsilon^2)$ samples for any $C > 0$.

\end{enumerate}

The rest of this paper is organized as follows: Section 2 reviews the learning framework and establishes generalization guarantees by analyzing the pseudo-dimension of the algorithm classes. Section 3 summarizes some results from 
\citet{g:16}, introduces our new cost function, and provides the corresponding generalization guarantees for the gradient descent method. Section 4 extends these generalization guarantees to the conjugate gradient method under specific assumptions. Finally, Section 5 concludes the paper.


\section{Background of learning complexity}


In the algorithm selection model, the primary approach involves choosing the best algorithm for a poorly understood application domain by learning the optimal algorithm based on an unknown instance distribution. We adopt the learning complexity framework for data-driven algorithm design introduced by 
\citet{g:16} to analyze the sample complexity bounds for tuning the step sizes of the gradient descent and conjugate gradient algorithms.
We briefly review the learning-theoretic framework, which consists of the following four components:
\begin{itemize}
    \item 
A fixed computational or optimization problem $\Pi$,
 \item An unknown distribution $\mathcal{D}$ over instances $x \in \Pi$,
 \item A set $\mathcal{A}$ of algorithms for solving $\Pi$,
 \item A cost function $c$: $\mathcal{A} \times \Pi \to [0,H]$, measuring the performance of a given algorithm on a given instance. 
\end{itemize}
The choice of cost function is pivotal in determining an algorithm’s performance. Various cost functions apply different criteria to identify the optimal algorithm.

The primary goal of algorithm selection/learning is to identify an algorithm $\mathit{A}_\mathcal{D}$ that minimizes $\mathbf{E}_{x \sim \mathcal{D}}[c(\mathit{A},x)]$ over $\mathit{A}\in \mathcal{A}$.
Learning complexity theory examines the efficiency of the learned algorithm. A commonly used measurement is the concept of $(\epsilon,\delta)$-learning with probability.








\begin{definition}[$(\epsilon,\delta)$-learning with probability {\rm \citep{k:94}}]
A learning algorithm $L$ $(\epsilon,\delta)$-learns the optimal algorithm in $\mathcal{A}$ from m samples if, for every distribution $\mathcal{D}$ over $\Pi$, with probability at least $1-\delta$ over m samples $x_1,x_2,\dots, x_m\sim \mathcal{D}$, L outputs an algorithm $\hat{A}\in \mathcal{A}$ with error at most $\epsilon$.
\end{definition}

Learning complexity essentially measures the performance of an algorithm in out-of-sample scenarios, often referred to as generalization guarantees. It can be estimated using the pseudo-dimension approach applied to the algorithm classes. 

\begin{definition}[Pseudo-dimension of the algorithm class {\citep{p:84}}]
Let $\mathcal{H}$ denote a set of real-valued functions defined on the set $X$. A finite subset $S=\{x_1,x_2,\dots, x_m\}$ of $X$ is shattered by $\mathcal{H}$ if there exist real-valued witnesses $r_1,r_2,\dots, r_m$ such that, for each of the $2^m$ subsets $T$ of $S$, there exists a function $h \in \mathcal{H}$, such that $h(x_i)>r_i$ if and only if $i \in T$ (for $i=1,2,\dots,m$). The pseudo-dimension of  $\mathcal{H}$ is the cardinality of the largest subset shattered by  $\mathcal{H}$.
\end{definition}
It is worth noting that the pseudo-dimension of a finite set $A$ is always at most $\log_2|A|$. To accurately estimate the expectation of all functions in  $\mathcal{H}$ with respect to any probability distribution $\mathcal{D}$ on $\mathcal{X}$, it is sufficient to bound the pseudo-dimension of $\mathcal{H}$.

\begin{theorem}[Uniform convergence {\citep{a:99}}]
Let $\mathcal{H}$ be a class of functions with domain $X$ and range in $[0,H]$, and suppose $\mathcal{H}$ has pseudo-dimension $d_\mathcal{H}$. For every distribution $\mathcal{D}$ over X, every $\epsilon>0$, and every $\delta \in (0,1]$, if
\begin{displaymath}
m \geq k(\frac{H}{\epsilon})^2(d_\mathcal{H}+\ln(\frac{1}{\delta}))
\end{displaymath}
for a suitable constant $k$ , then with probability at least $1-\delta$ over $m$ samples $x_1,x_2,\dots, x_m\sim \mathcal{D}$,
\begin{displaymath}
\Big|(\frac{1}{m}\sum_{i=1}^m h(x_i))-\mathbf{E}_{x \sim \mathcal{D}}[h(x)]\Big|< \epsilon,
\end{displaymath}
for every $h \in \mathcal{H}$.
\end{theorem}

For the class $\mathcal{A}$, considered as a set of real-valued functions defined on $\Pi$, we analyze its pseudo-dimension as defined above. Finally, we provide one additional definition.
\begin{definition}[Empirical risk minimization (ERM) {\citep{v:91}}]
Fix an optimization problem $\Pi$, a performance measure $c$, and a set of algorithms $\mathcal{A}$. An algorithm $\mathit{L}$ is an empirical risk minimization (ERM) algorithm if, given any finite subset S of $\Pi$, $\mathit{L}$ returns an algorithm from $\mathcal{A}$ with the best average performance on S.
\end{definition}

The concept of empirical risk minimization can be employed to ensure the existence of an algorithm that can be $(\epsilon,\delta)$-learned with probability.
Specifically, 
\citet{g:16} demonstrated that in any algorithm set with pseudo-dimension $d$, one can $(2\epsilon,\delta)$-learn the optimal algorithm with high probability, given a sufficiently large number of samples.
\begin{corollary}\rm{\citep{g:16}}
Fixing parameters $\epsilon>0$, $\delta \in (0,1]$, a set of problem instances
$\Pi$, and a cost function $c$. Let $\mathcal{A}$ be a set of algorithms that has pseudo-dimension $d$ with respect to $\Pi$ and $c$. Then any ERM algorithm $(2\epsilon,\delta)$-learns the optimal algorithm in $\mathcal{A}$ from $m$ samples, where $m$ is defined as in Theorem 1.
\end{corollary}

\section{Learning complexity of gradient descent algorithm}

\subsection{Learning complexity on number of iterations for gradient descent algorithm}


We apply the aforementioned learning complexity theory to the problem of learning step sizes from samples in an algorithm selection problem for unconstrained optimization. Initially, we restrict the algorithm set to gradient descent algorithms with varying step sizes.

Recall the basic gradient descent algorithm for minimizing a function $f$ given an initial point $z_0$ over $\mathbb{R}^N$:
\begin{enumerate}
  \item Initialize $z:=z_0$,
  \item While $||\nabla f(z)||_2>\nu$, $z:=z-\rho \nabla f(z)$.
\end{enumerate}
where the step size $\rho$ is the parameter of interest. For example, larger values of $\rho$ can make bigger progress per step but also carry the risk of overshooting a minimum of $f$.

We make the following assumptions on the gradient descent algorithm for unconstrained optimization problems \citep{g:16}.The reasoning and further elaboration regarding Assumption 1.4, along with Assumption 2.4 discussed in the next section, are detailed in the appendix.

\begin{assumption}
\begin{enumerate}
  \item[1)] $f$ is convex and $L$-smooth, i.e., for any $z_1, z_2 \in \mathbb{R}^N$, $||\nabla f(z_1)-\nabla f(z_2)|| \leq L||z_1-z_2||$,
  \item[2)] The initial points are bounded with $\nu <||z_0||\leq Z$,
  \item[3)] The step size $\rho$ is drawn from some interval: $\rho \in [\rho_l, \rho_u] \subset (0,\infty)$, 
  \item[4)] There exists a constant $\beta \in (0,1)$ such that $||z-\rho \nabla f(z)|| \leq (1-\beta)||z||$ for all $\rho \in [\rho_l, \rho_u]$.
\end{enumerate}
\end{assumption}

We define that $g(z,\rho):=z-\rho \nabla f(z)$ as the one-step iteration starting at $z$ with step size $\rho$. $g^j(z,\rho)$ represents the result after $j$ iterations. Based on the above assumptions, we have
\begin{displaymath}
||g^j(z,\rho)|| \leq (1-\beta)^j ||z||
\end{displaymath}
and
\begin{displaymath}
||\nabla f(g^j(z,\rho))|| \leq L(1-\beta)^j ||z||.
\end{displaymath}

We first review some {theoretical results derived by 
\citet{g:16},}
where a cost measure $c(A_\rho,x)$ is defined as the number of iterations required by the algorithm for the instance $x$.
Since $||z_0|| \leq Z$ and the algorithm stops once the gradient is less than $\nu$, it follows that $c(A_\rho,x)\leq \log(\nu/LZ)/\log(1-\beta)$ for all $\rho$ and $x$. Let $H = \log(\nu/LZ)/\log(1-\beta)$.  
We then estimate the error bound of the cost function with respect to different step sizes of a gradient decent algorithm: 
\begin{theorem}[Error estimation of cost function {\citep{g:16}}]
For any instance $x$, step sizes $\rho\in [\rho_l, \rho_u]$ and $\eta\in [\rho_l, \rho_u]$ with $0 \leq \eta - \rho \leq \frac{\nu \beta^2}{LZ}D(\rho)^{-H}$, where $D(\rho):=\max\{1,L\rho-1\}$, we have 
$|c(A_\rho,x)-c(A_\eta,x)|\leq 1$.
\end{theorem}
The following lemma provides an estimate of the error bound for iteration processes using different step sizes, which will be used
in the proofs in the next section.
\begin{lemma}[Error estimation of iteration processes {\citep{g:16}}]
For any $z \in \mathbb{R}^N$, $j>0$, and step sizes $\rho \leq \eta$,
\begin{displaymath}
    ||g^j(z,\rho)-g^j(z,\eta)|| \leq (\eta-\rho)\frac{D(\rho)^j LZ}{\beta}.
\end{displaymath} 
\end{lemma}
Moreover, let $K=\frac{\nu \beta^2}{LZ}D(\rho_u)^{-H}$. We choose $N$ as a minimal $K$-net, consisting of all integer multiples of $K$ that lie in $[\rho_l,\rho_u]$. Then $|N| \leq \rho_u/K+1$.    
It is known that the pseudo-dimension of $\mathcal{A}_N=\{\mathit{A}_\rho, \rho \in N\}$ is at most $\log|N|$. Since $\mathcal{A}_N$ is finite and $\log|N|\sim \tilde{O}(H)$, it also 
admits an ERM algorithm $L_N$, which $(\epsilon,\delta)$-learns the optimal algorithm in $\mathcal{A}_N$ using $m=\tilde{O}(H^2\log|N|/\epsilon^2)$ samples, where $\tilde{O}(\cdot)$ suppresses logarithmic factors in $Z/\nu$, $\beta$, $L$ and $\rho_u$.
The following theorem presents the result.

\begin{theorem}[Generalization Guarantees for GD on number of iterations  {\citep{g:16}}]
There is a learning algorithm that $(1+\epsilon, \delta)$-learns the optimal algorithm in $\mathcal{A}$ using $m=\tilde{O}(H^3/\epsilon^2)$ samples from $\mathcal{D}$.
\end{theorem}

\subsection{A new cost function}
Theorem 2 and Theorem 3 use the number of iterations in gradient descent as the cost function. However, it is challenging to extend this approach to more complex methods like CG. Moreover, when the optimization problem is too large or computation time is constrained, determining the number of iterations becomes impractical. Additionally, since the number of iterations must be an integer, the learning error $1+\epsilon$ cannot be reduced.

Considering the above facts, we introduce an alternative cost measure. In mixed integer linear programming (MILP), the primal-dual integral was first proposed in \citep{b:13}, aiming to determine 
the quality of best feasible solutions 
and 
the time when they are found during the solving process.
Specifically, for an MILP instance minimizing $\mathbf{c}^{\top}\mathbf{x}$ subject to $
\mathbf{A}^{\top}\mathbf{x} \leq \mathbf{b}$ and $\mathbf{x} \in \mathbb{Z}^p \times \mathbb{R}^{n-p}$,
the primal integral and dual integral are respectively defined as follows \citep{b:13}:
\begin{displaymath}
\begin{aligned}
P(T)&:=\int_{t=0}^{T} \mathbf{c}^{\top}\mathbf{x}_t^{*}dt - T\mathbf{c}^{\top}\mathbf{x}^{*},\\
D(T)&:=T\mathbf{c}^{\top}\mathbf{x}^{*}-\int_{t=0}^{T} \mathbf{z}_t^{*}dt,
\end{aligned}
\end{displaymath}
where $\mathbf{x}_t^{*}$ is the best feasible solution found at time $t$, $\mathbf{z}_t^{*}$ is the best dual bound at time $t$, and $T\mathbf{c}^{\top}\mathbf{x}^{*}$ is an instance-specific constant that depends on the optimal solution $x^{*}$.
Theoretically, both the primal integral and the dual integral are to be minimized, and takes an optimal value of 0, and a small dual integral will lead to both good and fast decisions \citep{g:22}. Using primal and dual integrals to measure the performance of MILPs has been widely used recent years \citep{g:22,q:22,z:23}.
Inspired by the concept of primal-dual integral in MILP, we define
\begin{displaymath}
c(\mathit{A}_\rho,x)=\sum_{j=1}^{M}||z^{*}-g^j(z_0,\rho)||,
\end{displaymath}
where $M$ is the number of iterations. Since gradient descent is translation invariant and $f$ is convex, we introduce the following conventions for iteration augmentation for convenience:
\begin{enumerate}
  \item $z^*=0$ and $f(z^*)=0$,
  \item For $j>M$, $g^j(z_0,\rho) = z^{*}$ and correspondingly, $||z^*-g^j(z_0,\rho)||=0$.
\end{enumerate}
Then $c(A_\rho ,x) \leq M||z_0|| \leq Z\log(\nu/LZ)/\log(1-\beta)$ for all $\rho$ and $x$, meaning that the cost measure is also bounded. 
Suppose at each iteration, the step size $\rho$ chosen in $N$ adheres to Assumption 1, then $\mathcal{A}_N =\{A_\rho: \rho \in N\}$ is also finite, and therefore its pseudo-dimension is at most $[\log(\nu/LZ)/\log(1-\beta)]\log|N|$.

\subsection{Learning complexity of gradient descent algorithm under the new cost function}

It is evident that, similar to the primal-dual integral, the new cost function offers several advantages. It can be calculated even if the iteration is halted before reaching the optimal value (for instance, when computation time exceeds a specified tolerance); a lower cost function generally signifies an algorithm that is both efficient and fast; among algorithms with the same number of iterations, the new cost function can discern the most effective one, theoretically corresponding to the algorithm with the highest convergence speed.

Drawing on the new cost function, we are able to establish new learning complexity results for the gradient descent algorithm.
Consider the error $|c(A_\rho ,x)-c(A_\eta ,x)|$, we have the following conclusion:
\begin{theorem}[Error estimation of new cost function]
For any constant $C > 0$, 
instance $x$, and 
step sizes $\rho\in [\rho_l, \rho_u]$, $\eta\in [\rho_l, \rho_u]$ with $0 \leq \eta - \rho \leq \frac{\beta}{LZ}\frac{1-D(\rho)}{1-D(\rho)^H}D(\rho)^{-1}C$, we have
\begin{displaymath}
  |c(A_\rho,x)-c(A_\eta,x)|\leq C.  
\end{displaymath}
\end{theorem}

Let $K=\frac{\beta}{LZ}\frac{1-D(\rho_u)}{1-D(\rho_u)^H}D(\rho_u)^{-1}C$. We choose $N$ as a minimal $K-$net, consisting of all integer multiples of $K$ within $[\rho_l,\rho_u]$. Then $|N| \leq \rho_u/K+1$. Consequently, by uniform convergence theorem and Corollary 1, an ERM algorithm can $(C+\epsilon, \delta)$-learn the optimal step size with sufficient number of samples. 
The following theorem demonstrates this result.

\begin{theorem}[Generalization guarantees for GD on new cost function]
For any constant $C > 0$, there exists a learning algorithm that $(C+\epsilon, \delta)$-learns the optimal algorithm in $\mathcal{A}$ using $m=\tilde{O}(H^3/\epsilon^2)$ samples from $\mathcal{D}$.
\end{theorem}

Compared with the learning complexity results based on number of iterations \citep{g:16}, we improved the learning complexity from  $(1+\epsilon, \delta)$-learns the optimal algorithm to $(C+\epsilon, \delta)$-learns the optimal algorithm in $\mathcal{A}$ using $m=\tilde{O}(H^3/\epsilon^2)$ samples.



\section{Learning complexity of conjugate gradient algorithm}

Based on the new cost function, we are able to investigate the learning complexity of the conjugate gradient algorithm focusing on both the step size and the conjugate parameter.
Recall the basic conjugate gradient algorithm for minimizing a function $f$ given an initial point $z_0$ over $\mathbb{R}^N$:
\begin{enumerate}
  \item Initialize $z:=z_0$,
  \item For $n \geq 1$, while $||\nabla f(z_n)||_2>\nu$, $z_1:=z_{0}-\rho \nabla f(z_0)$, $z_{n+1}:=z_n - \rho \nabla f(z_n)-\eta (z_n-z_{n-1})$.
\end{enumerate}
where $\rho$ and $\eta$ denote the step size and the conjugate parameter (we simply call both of them the step sizes). 
We denote the $j-$step iteration from $z_0$ and $z_1$ with step sizes $\rho$ and $\eta$ as $g^j(z_1,z_0,\rho,\eta)=z_j - \rho \nabla f(z_j)-\eta (z_j-z_{j-1})$. The $j$-th step starting from  $z_0$ and $z_1$ can be viewed as the first step starting from  $z_j$ and $z_{j-1}$, thus $g^j(z_1,z_0,\rho,\eta)=g^1(z_j,z_{j-1},\rho,\eta)$. For convenience, we define that $g(z_j,z_{j-1},\rho,\eta):=g^1(z_j,z_{j-1},\rho,\eta)$ and $g^0(z_1,z_0,\rho,\eta)=z_1$.

We make the following assumptions about the conjugate gradient algorithm for unconstrained optimization problems.
\begin{assumption}

\begin{enumerate}
  \item[1)] $f$ is convex and $L$-smooth, i.e. for any $z_1$ and $z_2$, $||\nabla f(z_1)-\nabla f(z_2)|| \leq L||z_1-z_2||$,
  \item[2)] The initial points are bounded with $\nu <||z_0||\leq Z$,
  \item[3)] The step sizes $\rho$ and $\eta$ are drawn from some interval $\rho \in [\rho_l, \rho_u] \subset (0,\infty)$ and $\eta\in [\eta_l, \eta_u] \subset (0,\infty)$ respectively,
  \item[4)] There exists a constant $\beta \in (0,1)$ such that for any $z_n$ and $z_{n-1}$, $||z_n - \rho \nabla f(z_n)-\eta (z_n-z_{n-1})|| \leq (1-\beta)||z_n||$ for all $\rho \in [\rho_l, \rho_u]$ and  $\eta\in [\eta_l, \eta_u] \subset (0,\infty)$ respectively.
\end{enumerate}
\end{assumption}
If the condition 4) in Assumption 2 holds, then $||z_{n+1}|| \leq (1-\beta)||z_n||$. Given the arbitrariness of $n$, it follows that  $||z_{n}|| \leq (1-\beta)||z_{n-1}|| $. By induction, this implies $||z_{n}|| \leq (1-\beta)^n||z_0||$, and consequently, $||g^j(z_i,z_{i-1},\rho,\eta)|| \leq (1-\beta)^j||z_i||$. Then by condition 1) in Assumption 2, we have $||\nabla g^j(z_1,z_{0},\rho,\eta)|| \leq L(1-\beta)^j||z_0||$. According to the termination condition $||\nabla f(z_n)||_2\leq \nu$, the number of iterations is less than $H:=\log(LZ/\nu)/\log(1-\beta)$.

To establish the desired results, we first need to utilize some results related to second order homogeneous linear recurrence relations.
\begin{definition}[Second order homogeneous linear recurrence relation]
Any recurrence relation of the form
\begin{displaymath}
x_n = ax_{n-1}+bx_{n-2}
\end{displaymath}
is a second-order homogeneous linear recurrence relation.
\end{definition}
The general formula for any second-order homogeneous linear recurrence relation is given in the following lemma.
\begin{lemma}[General formula {\citep{s:10}}]
The general solution for $x_n = ax_{n-1}+bx_{n-2}$ is given by
\begin{displaymath}
x_n = c_1r_1^n+c_2r_2^n,
\end{displaymath}\small
where $r_1$ and $r_2$ are the roots of the characteristic equation
\begin{displaymath}
r^2-ar-b=0,
\end{displaymath}
$c_1$ and $c_2$ are 
constants by plugging in given $x_0$ and $x_1$.
\end{lemma}

We first demonstrate the 
Lipschitz continuity of $g$.
\begin{lemma}[Lipschitz continuity of $g$] For any $w_n$, $w_{n-1} \in \mathbb{R}^N$ in a CG iteration process (resp. $y_n$, $y_{n-1} \in \mathbb{R}^N$), and any step sizes $\rho \in [\rho_l, \rho_u]$, $\eta \in [\eta_l, \eta_u]$,
\begin{displaymath}
\begin{split}
&||g(w_n,w_{n-1},\rho,\eta)-g(y_n,y_{n-1},\rho,\eta)|| \\
\leq & [D(\rho)+\eta]||w_{n}-y_{n}||+\eta||w_{n-1}-y_{n-1}||,
\end{split}
\end{displaymath}
where $D(\rho):=max\{1,L\rho-1\}$.
\end{lemma}

We extend Lemma 3 to the case of $g^j$ and derive the following result.
\begin{corollary}[Lipschitz continuity of $g^j$]
For any $w_n$, $w_{n-1} \in \mathbb{R}^N$ in a CG iteration process started at $w_0$ (resp. $y_n$, $y_{n-1} \in \mathbb{R}^N$ started at $y_0$), and any step sizes $\rho \in [\rho_l, \rho_u]$, $\eta \in [\eta_l, \eta_u]$, 
\begin{displaymath}
\begin{split}
& ||g(w_n,w_{n-1},\rho,\eta)-g(y_n,y_{n-1},\rho,\eta)|| \\
= & ||g^n(w_1,w_{0},\rho,\eta)-g^n(y_1,y_{0},\rho,\eta)||\\
\leq &\left[\frac{D(\rho)-r_2}{r_1-r_2}r_1^n+\frac{D(\rho)-r_1}{r_2-r_1}r_2^n\right]||w_{0}-y_{0}||\\
:=&F(\rho,\eta,n)||w_{0}-y_{0}||,
\end{split}
\end{displaymath}
where $r_1$ and $r_2$ are the roots of characteristic equation of second order homogeneous linear recurrence relation $x_n = [D(\rho)+\eta]x_{n-1}+\eta x_{n-2}$.
\end{corollary}

{
We aim to derive a bound on how far two conjugate gradient iteration processes, with different algorithm parameters 
$\rho$ and $\eta$, can derive from each other when they start from the same point.
In detail, we assume that the two algorithms use the same starting point $z_0$, 
but different step sizes 
$(\rho_1,\eta_1)$  and $(\rho_2,\eta_2)$ respectively at any iteration step $j$.
Our goal is to estimate the upper bound of $||g^j(z_1,z_0,\rho_1,\eta_1)-g^j(z^{\prime}_1,z_0,\rho_2,\eta_2)||$. Here $z^{\prime}_1$ corresponds to another iteration processes different from $z_1$ due to a different choice of step sizes $\rho_2$, $\eta_2$.
}

To obtain the estimation, we first consider two special cases. The first case involves fixing the step size $\eta$ and examining the error only when $\rho$ changes.
\begin{lemma}[Iteration error estimation fixing the step size $\eta$]
For any initial point $z_0\in \mathbb{R}^N$, iteration step $j \geq 2$, step sizes $\eta\in \mathbb{R}$, $\rho_1\in \mathbb{R}$, $\rho_2\in \mathbb{R}$ such that $\rho_1 \leq \rho_2$, we have 
\begin{displaymath}
\begin{aligned}
&||g^j(z_1,z_0,\rho_1,\eta)-g^j(z^{\prime}_1,z_0,\rho_2,\eta)||\\ 
\leq & (\rho_2-\rho_1)\bigg[(\frac{R_0r_2-R_1}{r_2-r_1}r_1^j+\frac{R_0r_1-R_1}{r_1-r_2}r_2^j)\\
&-\frac{LZ(1-\beta)^{j+2}}{[D(\rho_1)+\eta](1-\beta)+\eta-(1-\beta)^2}\bigg]\\
:=& (\rho_2-\rho_1)G(\rho_1,\eta,j),
\end{aligned}
\end{displaymath}
where $D(\rho):=\max\{1,L\rho-1\}$, $r_1$ and $r_2$ are the roots of the characteristic equation of second order homogeneous linear recurrence relation $A_n = [D(\rho_1)+\eta]A_{n-1}+\eta A_{n-2}$,
\begin{displaymath}
\begin{aligned}
&R_0 = \frac{LZ(1-\beta)^2}{[D(\rho_1)+\eta](1-\beta)+\eta-(1-\beta)^2},\\
&R_1 = \frac{D(\rho_1)LZ}{\beta}+\frac{LZ(1-\beta)^3}{[D(\rho_1)+\eta](1-\beta)+\eta-(1-\beta)^2}.\\
\end{aligned}
\end{displaymath} 
\end{lemma}

Next, we consider the case where the step size $\rho$ is fixed and analyze the error estimation only when $\eta$ changes. The result is similar and thus we present it in the appendix as Lemma 5.

With the above two results, we can address the general case where both $\rho$ and $\eta$ change, and furthermore, we have
\setcounter{lemma}{5}
\begin{lemma}[Iteration error estimation]
For any initial point $z_0\in \mathbb{R}^N$, iteration step $j \geq 2$, step sizes $\rho_1 \in \mathbb{R}$, $\rho_2 \in \mathbb{R}$ $\eta_1\in \mathbb{R}$, $\eta_2\in \mathbb{R}$, we have 
\begin{displaymath}
\begin{split}
&||g^j(z_1,z_0,\rho_1,\eta_1)-g^j(z^{\prime}_1,z_0,\rho_2,\eta_2)|| \\
\leq & |\rho_2-\rho_1|G(\rho_1,\eta_1,j)+|\eta_2-\eta_1|H(\eta^*,\rho_1,j),
\end{split}
\end{displaymath}
where further details of $H(\eta^*,\rho_1,j)$ are present it in  Lemma 5 in the appendix.
\end{lemma}
We then get the following corollaries.
\begin{corollary}
\begin{itemize}
\item[1)] For any $C > 0$, if $\rho$ is invariant in the whole iteration process and $|\eta_2-\eta_1| \leq \frac{C}{H(\eta^*,\rho,j)}$, then the iteration error satisfies $||g^j(z_1,z_0,\rho,\eta_1)-g^j(z^{\prime}_1,z_0,\rho,\eta_2)|| \leq C$.
\item[2)] For any $C > 0$, if $\eta$ is invariant during the whole iteration process, and $|\rho_2-\rho_1| \leq \frac{C}{G(\eta,\rho_1,j)}$, then the iteration error satisfies $||g^j(z_1,z_0,\rho_1,\eta)-g^j(z^{\prime}_1,z_0,\rho_2,\eta)|| \leq C$.
\item[3)] More generally, for any $C > 0$, if the parameter combination satisfies $|\rho_2-\rho_1|G(\rho_1,\eta_1,j)+|\eta_2-\eta_1|H(\eta^*,\rho_1,j) \leq C $, then we have  $||g^j(z_1,z_0,\rho_1,\eta_1)-g^j(z^{\prime}_1,z_0,\rho_2,\eta_2)|| \leq C $.
\end{itemize}
\end{corollary}

Finally, we present a more general theorem concerning the learning complexity of the conjugate gradient algorithm with respect to the algorithm parameters. This result follows directly from Lemma 6.
\begin{theorem}[Iteration error estimation]
For any $z_0$, $j \geq 2$, $\rho_1$, $\rho_2$, $\eta_1$, and $\eta_2$ with $|\rho_2-\rho_1|<\varepsilon_1$ and $|\eta_2-\eta_1|<\varepsilon_2$,
\begin{displaymath}
\begin{split}
&||g^j(z_1,z_0,\rho_1,\eta_1)-g^j(z^{\prime}_1,z_0,\rho_2,\eta_2)||\\ 
\leq & \varepsilon_1 G(\rho_1,\eta_1,j)+\varepsilon_2 H(\eta^*,\rho_1,j),
\end{split}
\end{displaymath}
where 
$z^{\prime}_1$ corresponds to another iteration processes different from $z_1$ due to a different choice of step sizes $\rho_2$, $\eta_2$.
Therefore, $||g^j(z_1,z_0,\rho_1,\eta_1)-g^j(z^{\prime}_1,z_0,\rho_2,\eta_2)||$ converges to zero if $\varepsilon_1$ and $\varepsilon_2$ tend to zero.
\end{theorem}

We next refine our cost function. We can also define $c(A_{\rho, \eta},x)=\sum_{i=0}^{M}||z^{*}-g^i(z_1,z_0,\rho,\eta)||$, where $A_{\rho,\eta}$ stands for the conjugate gradient algorithm with step sizes $\rho$ and $\eta$. Note that $M<H$, also for convenience, we introduce the following conventions:
\begin{enumerate}
  \item $z^*=0$ and $f(z^*)=0$,
  \item For $j>M$, $g^j(z_1,z_0,\rho,\eta) = z^{*}$ and correspondingly, $||z^*-g^j(z_1,z_0,\rho,\eta)||=0$.
\end{enumerate}

Then $c(A_{\rho,\eta} ,x) \leq M||z_0|| \leq Z\log(\nu/LZ)/\log(1-\beta)$ for all $\rho$, $\eta$ and $x$. This means the cost measure is bounded.\\


Consider the error in the cost function between different algorithms, measured by $|c(A_{\rho_1,\eta_1} ,x)-c(A_{\rho_2,\eta_2} ,x)|$, we have the following theorem.
\begin{theorem}[Error estimation of cost function]
For any instance $x$, any step sizes $\rho_1$, $\rho_2$, $\eta_1$ and $\eta_2$, and a constant $C > 0$, if
\begin{displaymath}
\begin{split}
&|\rho_2-\rho_1|\frac{LZD(\rho_1)}{\beta}+\sum_{j=1}^{M}\bigg[G(\rho_1,\eta_1,j)]\\
&\quad +|\eta_2-\eta_1|\sum_{j=1}^{M}[H(\eta^*,\rho_1,j)\bigg] \leq C,
\end{split}
\end{displaymath}
Then
\begin{displaymath}
|c(A_{\rho_1,\eta_1} ,x)-c(A_{\rho_2,\eta_2} ,x)|\leq C.
\end{displaymath}
\end{theorem}


We readily derive the following corollary, which provides a sufficient condition for Theorem 7. 
\begin{corollary}
For any constant $C>0$, instance $x$, step sizes $\rho_1$, $\rho_2$, $\eta_1$, $\eta_2$ with
\begin{displaymath}
  |\rho_2-\rho_1|\left[\frac{LZD(\rho_1)}{\beta}+\sum_{j=1}^{M}G(\rho_1,\eta_1,j)\right]\leq C/2,
\end{displaymath}
  and 
\begin{displaymath}
    |\eta_2-\eta_1|\sum_{j=1}^{M}H(\eta^*,\rho_1,j)\leq C/2,
\end{displaymath}
we have
\begin{displaymath}
|c(A_{\rho_1,\eta_1} ,x)-c(A_{\rho_2,\eta_2} ,x)|\leq C.
\end{displaymath}
\end{corollary} 
In the remainder of this section, for convenience, we assume $\rho_1 \leq \rho_2$ and $\eta_1 \leq \eta_2$. 
For $\rho_1$ and $\rho_2$ (and similarly for $\eta_1$ and $\eta_2$), due to 
$(\rho_2-\rho_1)[\frac{LZD(\rho_1)}{\beta}+\sum_{j=1}^{M}G(\rho_1,\eta_1,j)]\leq C/2,$ which implies that if we find an upper bound of $\frac{LZD(\rho_1)}{\beta}+\sum_{j=1}^{M}G(\rho_1,\eta_1,j)$, denoted as $G^*$, then Corollary 4 holds if $\rho_2-\rho_1 \leq \frac{C}{2G^*}$. Thus, the problem reduces to finding an appropriate $G^*$. One straightforward method to obtain $G^*$ is to compute $G(\rho_1,\eta_1,j)$ for all $j=1$ to $M$. If we let the maximum of these $M$ values be $\tilde{G}$, then $G^* = \frac{LZD(\rho_1)}{\beta}+H\tilde{G}$. The complexity of calculating this $G^*$ is $O(Hn^H)$, where $n$ is a constant and $n^H$ measures the complexity of computing each $G(\rho_1,\eta_1,j)$. The following lemma provides a refined estimation of the upper bound $G^*$ to reduce the complexity.

\begin{lemma}[A refined upper bound for $G$]
\begin{displaymath}
G^* = \frac{LZD(\rho_1)}{\beta}+ R^*\left[r_1\frac{1-r_1^H}{1-r_1}\right]\\
\end{displaymath}
is an upper bound for $\frac{LZD(\rho_1)}{\beta}+\sum_{j=1}^{M}G(\rho_1,\eta_1,j)$,
where $r_1 \geq r_2$ are the roots of the characteristic equation of second order homogeneous linear recurrence relation $A_n = [D(\rho_1)+\eta_1]A_{n-1}+\eta_1 A_{n-2}$, and
\begin{displaymath}
\begin{aligned}
R_0 = &\frac{LZ(1-\beta)^2}{[D(\rho_1)+\eta_1](1-\beta)+\eta_1-(1-\beta)^2},\\
R_1 = &\frac{LZD(\rho_1)}{\beta}+\frac{LZ(1-\beta)^3}{[D(\rho_1)+\eta_1](1-\beta)+\eta_1-(1-\beta)^2},\\
R^* = &\frac{R_0r_2-R_1}{r_2-r_1}+\Big|\frac{R_0r_1-R_1}{r_1-r_2}\Big|. \\
\end{aligned}
\end{displaymath}
\end{lemma}

Similarly, for $\sum_{i=1}^{M}[H(\eta^*,\rho_1,j)]$ and its upper bound $H^*$, we have
\begin{lemma}[A refined upper bound for $H$]
\begin{displaymath}
H^* = R^{*'}\left[r_1\frac{1-r_1^H}{1-r_1}\right]
\end{displaymath}
is an upper bound for $\sum_{i=1}^{M}[H(\eta^*,\rho_1,i)]$,
where 
$r_1 \geq r_2$ are the roots to the characteristic equation of second order homogeneous linear recurrence relation $A_n = [D(\rho_1)+\eta_1]A_{n-1}+\eta_1 A_{n-2}$, and 
\begin{displaymath}
\begin{aligned}
R_0 = & \frac{F(\eta^*)}{D(\rho_1)+2\eta_1-1},\\
R_1 = & F(\eta^*)\frac{1}{\eta_1}\left[1+\frac{1}{D(\rho_1)+2\eta_1-1}\right],\\
F(\eta^*)=& \mathop{\rm{max}}_{
        n \in(0,1,\cdots,j-1)
        \atop
        \eta \in [\eta_l,\eta_u]} \eta^{n}\left[\rho_1 LZ+\frac{\rho_1 Z}{\eta-L}L\right]-\frac{\rho_1 Z}{\eta-L}L^{n+1},\\
R^{*'} = &\frac{R_0r_2-R_1}{r_2-r_1}+\Big|\frac{R_0r_1-R_1}{r_1-r_2}\Big|.\\
\end{aligned}
\end{displaymath}
\end{lemma}

Lemma 7 and Lemma 8 demonstrate that for any given instance and for two distinct CG algorithms with different step sizes, the difference in their corresponding cost functions is bounded, if the difference between the step sizes is below a certain tolerance. Thus, we derive the following result:
\begin{theorem}[Error estimation of cost function]
For any constant $C>0$, any instance $x$, any step sizes $\rho_1$, $\rho_2$, $\eta_1$ and $\eta_2$ satisfying
\begin{displaymath}
\begin{aligned}
|\rho_2-\rho_1| \leq \frac{C}{2G^*},\quad |\eta_2-\eta_1| \leq \frac{C}{2H^*},
\end{aligned}
\end{displaymath}
then the difference between the corresponding cost functions is bounded by
\begin{displaymath}
|c(A_{\rho_1,\eta_1} ,x)-c(A_{\rho_2,\eta_2} ,x)|\leq C,
\end{displaymath}
where $G^*$ and $H^*$ are defined as those in Lemma 7 and Lemma 8.
\end{theorem}

We can also set a $K_\rho$-net (resp. a $K_\eta$-net) for $\rho$ (resp. for $\eta$). Suppose $N_{\rho}$ is a minimal $K_\rho$-net, consisting of all integer multiples of $K_\rho$ that lie in $[\rho_l,\rho_u]$. Then $|N_{\rho}| \leq \rho_u/K_\rho+1$ (and the same applies to $N_\eta$).
Let
\begin{displaymath}
\begin{aligned}
K_\rho = \mathop{\rm{max}}_{
        \rho \in [\rho_l,\rho_u]
        \atop
        \eta \in [\eta_l,\eta_u]} \frac{C}{2G^*(\rho,\eta)},\ \ 
K_\eta =  \mathop{\rm{max}}_{
        \rho \in [\rho_l,\rho_u]
        \atop
        \eta \in [\eta_l,\eta_u]} \frac{C}{2H^*(\rho,\eta)},\\
\end{aligned}
\end{displaymath}
then $\log|N_{\rho}|\sim \tilde{O}(H)$ and $\log|N_{\eta}|\sim \tilde{O}(H)$.

Finally, by uniform convergence theorem and Corollary 1, an ERM algorithm can $(C+\epsilon,\delta)$-learn the optimal step size with a sufficient number of samples. The following theorem demonstrates this result.
\begin{theorem}[Generalization Guarantees for CG]
For any constant $C>0$, there exists a learning algorithm that $(C+\epsilon, \delta)$-learns the optimal algorithm in $\mathcal{A}_{N_\rho,N_\eta}$ using $\tilde{O}(H^4/\epsilon^2)$ samples from $\mathcal{D}$.
\end{theorem}
By now, we have shown the existence a learning algorithm capable of probabilistically identifying the optimal algorithm with
a sufficiently large sample size by using the proposed new cost function under the learning complexity framework.

\section{Conclusion}
In this study, we employed a learning-complexity framework for data-driven algorithm design. One of our major contributions is to introduce a new cost function based on the sum of distances between current and optimal values for each iteration. We proved that for both gradient descent and conjugate gradient methods, optimal learning algorithms exist that can $(C+\epsilon,\delta)$-learn the optimal algorithm using $m=\tilde{O}(H^3/\epsilon^2)$ and $m=\tilde{O}(H^4/\epsilon^2)$ samples, respectively.

Future researches include: extend this framework to other optimization algorithms, develop new cost functions for better performance, or investigate the scalability for large-scale problems.  
These efforts will enhance the applicability and effectiveness of data-driven algorithm design in solving complex optimization challenges. In recent years, researchers have predominantly focused on the worst-case complexity of algorithms, emphasizing the derivation of upper bounds. In contrast, lower bounds on complexity are often applicable only to specific problems. We believe that exploring this area further presents an intriguing avenue for future research.
\bibliographystyle{aaai25}
\nocite{*}
\bibliography{aaai25}
\newpage

\appendix

\section{Notation Table}
\begin{table}[ht]
\centering
\begin{tabular}{|c|l|}
\hline
\textbf{Symbol} & \textbf{Meaning} \\ \hline
$\Pi$ & A fixed computational or optimization problem. \\ \hline
$\mathcal{D}$ & An unknown distribution over instances in $\Pi$. \\ \hline
$\mathcal{A}$ & algorithms for solving $\Pi$.\\ \hline
$c$ & cost function. \\ \hline
$\mathit{A}_\mathcal{D}$ & The 'best' algorithm in $\mathcal{A}$. \\ \hline
$d_\mathcal{H}$ & The pseudo-dimension a class of functions $\mathcal{H}$. \\ \hline
$f$ & A given function to solve by GD/CG method. \\ \hline
$z_0$ & The initial point in GD/CG iterations. \\ \hline
$z^*$ & The optimal solution for minimizing $f$. \\ \hline
$\rho$,$\eta$ & Step sizes in GD/CG method. \\ \hline
$\nu$ & The lower bound of the norm of the initial point $||z_0||$. \\ \hline
$\beta$ & A constant mentioned in Assumption 1.4 and Assumption 2.4\\ \hline
$g^j$ & the result after $j$ iterations starting at some initial point with some step size. \\ \hline
$A_\rho$,$A_{\rho,\eta}$ & The GD (CG) with step size $\rho$ ($\rho$ and $\eta$).\\ \hline
$D(\rho)$ & $\max\{1,L\rho-1\}$ \\ \hline
$\tilde{O}(\cdot)$ & Softened Big-O notation representing upper bound of an algorithm's complexity. \\ \hline
\end{tabular}
\caption{Notation Table}
\end{table}

\section{Details of Assumption 1.4}
Assumption 1.4 (the case of Assumption 2.4 is similar),
which states that {\it ``there exists a constant $\beta \in (0,1)$ such that 
$$||z-\rho \nabla f(z)|| \leq (1-\beta)||z|| \text{ for all } \rho \in [\rho_l, \rho_u]\text{''.}\qquad (1)$$} 
is a condition for the selection of proper step size such that the algorithm provide a sufficient decrease condition at each step. It plays a similar role to the traditional Armijo condition or Wolfe condition \citet{w:69}, but in a form more suitable for learning complexity analysis.

{\bf Motivation of the condition:}
At each step of the algorithm, $z$ in the condition stands for the current iteration point $z_t$ at step $t$ and $z-\rho \nabla f(z)$ is the next iteration point $z_{t+1}$ under the direction $\nabla f(z)$ and selected $\rho$. 
As we have assumed that the optimal solution $z^*=0$ (this assumption can be made by simply adding a constant to the decision variable and does not effect the iteration process), thus (1) essentially means $$||z_{k+1}-z^*|| \leq (1-\beta)||z_k-z^*||, \text{ for some } \beta\in[0,1].$$

{\bf Strongly convex problem case:}
The condition is not very strong to satisfy when we apply GD or CG algorithms in practical optimization problems.
For GD algorithm as an example, we consider the optimization problems when the objective function $f$ is $m$-strongly convex, which states that for any points $z_1$, $z_2 \in \mathbb{R}^n$, 
$$
f(z_2) + \nabla f(z_2)^\top(z_1-z_2) +\frac{m}{2}||z_1-z_2||^2 \leq f(z_1).$$
When $m = 0$, it reverts to the classical convexity.
We define $\kappa=\frac{L}{m}$, where $L$ is the Lipschitz modules of $f$ as defined in Assumption 1.1, and set the step size as $\rho = \frac{2}{m+L}$. With any starting point $z_0 \in \mathbb{R}^n$, we have
$$
\begin{aligned}
||z_{t+1}-z^*||^2 &= ||z_{t}-\rho \nabla f(z_t)-z^*||^2\\
&= ||z_{t}-z^*||^2+\rho^2||\nabla f(z_t)||^2-2\rho\nabla f(z_t)^\top(z_{t}-z^*) \\
&\leq \left(1-2\frac{\rho mL}{m+L}\right)||z_{t}-z^*||^2 + (\rho^2-2\frac{\rho}{m+L})||\nabla f(z_t)||^2 \\
&= \left(\frac{\kappa-1}{\kappa+1}\right)^2||z_{t}-z^*||^2.
\end{aligned}
$$
The inequality presented above is by Lemma 3.11 in \citet{b:15}. 
We can then observe that in this case, for all parameter $\beta \in (0,1-\frac{\kappa-1}{\kappa+1}]$, 
there exists at least a step size $\rho:=\frac{2}{m+L}$ satisfying (1).
Thus, we can find that the set of step sizes under Assumption 1.4 is non-empty in the $L$-smooth and $m$-strongly convex optimization problems.

Given $\beta \leq 1-\frac{\kappa-1}{\kappa+1}$, this condition (1)  provides a constraint for selecting an appropriate step size $\rho$. 
When $\beta$ is close to 0, this assumption reduces to the decrease condition and provides a wide range for the selection of step size.
When $\beta$ is larger, the algorithm has a better performance guarantee.

\section{Proof of Theorem 4}
\begin{proof}
Let's assume that $M_\rho$, the number of iterations for algorithm $A_\rho$ is larger than $M_\eta$, and we denote $M_\rho$ by $M$. By Lemma 1, $||g^j(z,\rho)-g^j(z,\eta)|| \leq (\eta-\rho)\frac{D(\rho)^j LZ}{\beta}$. Therefore, we have
\begin{displaymath}
\begin{aligned}
&|c(A_\rho,x)-c(A_\eta,x)|\\
=& |\sum_{i=1}^{M}||z^{*}-g^i(z_0,\rho)||-\sum_{i=1}^{M}||z^{*}-g^i(z_0,\eta)|||\\
=&\sum_{i=1}^{M}||z^{*}-g^i(z_0,\rho)||-\sum_{i=1}^{M}||z^{*}-g^i(z_0,\eta)||\\
=&\sum_{i=1}^{M}[||z^{*}-g^i(z_0,\rho)||-||z^{*}-g^i(z_0,\eta)||]\\
\leq& \sum_{i=1}^{M}[||g^i(z_0,\rho)||-||g^i(z_0,\eta)||]\\
\leq& \sum_{i=1}^{M}[||g^i(z_0,\rho)-g^i(z_0,\eta)||]\\
\leq& \sum_{i=1}^{M}[(\eta-\rho)\frac{D(\rho)^j LZ}{\beta}]\\
=& (\eta-\rho)\frac{LZ}{\beta}\frac{1-D(\rho)^M}{1-D(\rho)}D(\rho).
\end{aligned}
\end{displaymath}
Recall that $M \leq H$, plugging in the bound in assumption gives the desired result.\\
\end{proof}

\section{Proof of Theorem 5}
\begin{proof}
    We have known that the pseudo-dimension of $\mathcal{A}_N=\{\mathit{A}_\rho, \rho \in N\}$ is at most $[\log(\nu/LZ)/\log(1-\beta)]\log|N|$ in Section 3.2. Since $\mathcal{A}_N$ is finite and $\log|N|\sim \tilde{O}(H)$, it also trivially admits an ERM algorithm $L_N$, which $(\epsilon,\delta)$-learns the optimal algorithm in $\mathcal{A}_N$ using $m=\tilde{O}(H^3/\epsilon^2)$ samples.
\end{proof}

\section{Proof of Lemma 3}
\begin{proof} By Lemma 3.16 in  Gupta and Roughgarden (2016) and triangle inequality, we have that 
\begin{displaymath}
\begin{aligned}
&||g(w_n,w_{n-1},\rho,\eta)-g(y_n,y_{n-1},\rho,\eta)||\\
=&||[w_n-\rho\nabla f(w_n)-\eta(w_n-w_{n-1})]-[y_n-\rho\nabla f(y_n)-\eta(y_n-y_{n-1})]||\\
=&||(w_n-y_n)-\rho(\nabla f(w_n)-\nabla f(y_n))-\eta[(w_n-w_{n-1})-(y_n-y_{n-1})]||\\
\leq& D(\rho)||w_n-y_n||+\eta||(w_n-y_n)-(w_{n-1}-y_{n-1})||\\
\leq& [D(\rho)+\eta]||w_{n}-y_{n}||+\eta||w_{n-1}-y_{n-1}||.
\end{aligned}
\end{displaymath}
\end{proof}

\section{Proof of Corollary 2}
\begin{proof}
By Lemma 3,
\begin{displaymath}
\begin{aligned}
&||g(w_n,w_{n-1},\rho,\eta)-g(y_n,y_{n-1},\rho,\eta)|| \leq [D(\rho)+\eta]||w_{n}-y_{n}||+\eta||w_{n-1}-y_{n-1}||, 
\end{aligned}
\end{displaymath}
i.e. 
\begin{displaymath}
\begin{aligned}
||w_{n+1}-y_{n+1}|| \leq [D(\rho)+\eta]||w_{n}-y_{n}||+\eta||w_{n-1}-y_{n-1}||.
\end{aligned}
\end{displaymath}
Let $A_n=||w_{n}-y_{n}||$, we have
\begin{displaymath}
A_n = [D(\rho)+\eta]A_{n-1}+A_{n-2}.
\end{displaymath}
Therefore, by Lemma 1, 
\begin{displaymath}
A_n = c_1r_1^n+c_2r_2^n,
\end{displaymath}
where $r_1$ and $r_2$ are the roots of characteristic equation of second order homogeneous linear recurrence relation $A_n = [D(\rho)+\eta]A_{n-1}+\eta A_{n-2}$. Now consider the coefficients $c_1$ and $c_2$. Given that the initial value satisfies $A_0=||w_{0}-y_{0}||$, and $A_1=||w_{1}-y_{1}||\leq D(\rho)||w_{0}-y_{0}||$ (According to Lemma 3.16 in  Gupta and Roughgarden (2016)), and\\
\begin{displaymath}
\centering
\left \{
\begin{aligned}
A_0 &= c_1+c_2 \\
A_1 &= c_1r_1+c_2r_2. \\
\end{aligned}
\right.
\end{displaymath}
We solve the equations above to get
\begin{displaymath}
\centering
\left \{
\begin{aligned}
c_1 &= \frac{D(\rho)-r_2}{r_1-r_2}||w_{0}-y_{0}|| \\
c_2 &= \frac{D(\rho)-r_1}{r_2-r_1}||w_{0}-y_{0}|| \\
\end{aligned}
\right.
\end{displaymath}
which gives the corollary.
\end{proof}

\section{Proof of Lemma 4}

\begin{displaymath}
\begin{aligned}
&R_0 = \frac{LZ(1-\beta)^2}{[D(\rho_1)+\eta](1-\beta)+\eta-(1-\beta)^2},\\
&R_1 = \frac{D(\rho_1)LZ}{\beta}+\frac{LZ(1-\beta)^3}{[D(\rho_1)+\eta](1-\beta)+\eta-(1-\beta)^2}.\\
\end{aligned}
\end{displaymath}

\begin{proof} By Lemma 3, we have that 
\begin{displaymath}
\begin{aligned}
&||g^j(z_1,z_0,\rho_1,\eta)-g^j(z^{\prime}_1,z_0,\rho_2,\eta)||\\
=&||g(z_j,z_{j-1},\rho_1,\eta)-g(z^{\prime}_j,z^{\prime}_{j-1},\rho_2,\eta)||\\
=&||[z_{j}-\rho_1 \nabla f(z_{j})-\eta(z_{j}-z_{j-1})]-[z^{\prime}_j-\rho_2 \nabla f(z^{\prime}_j)-\eta(z^{\prime}_j-z^{\prime}_{j-1})]||\\
=&||[z_{j}-\rho_1 \nabla f(z_{j})-\eta(z_{j}-z_{j-1})]-[z^{\prime}_j-\rho_1 \nabla f(z^{\prime}_j)-\eta(z^{\prime}_j-z^{\prime}_{j-1})]+[(\rho_2-\rho_1)\nabla f(z^{\prime}_j)]||\\
\leq& [D(\rho_1)+\eta]||z_{j}-z^{\prime}_j||+\eta||z_{j-1}-z^{\prime}_{j-1}||+(\rho_2-\rho_1)||\nabla f(z^{\prime}_j)]||.\\
\end{aligned}
\end{displaymath}
By condition 1) in Assumption 2, $||\nabla f(z^{\prime}_j)]|| \leq L||z^{\prime}_j||\leq LZ(1-\beta)^j$. Then we have
\begin{displaymath}
\begin{aligned}
&||g^j(z_1,z_0,\rho_1,\eta)-g^j(z^{\prime}_1,z_0,\rho_2,\eta)||
\leq [D(\rho_1)+\eta]||z_{j}-z^{\prime}_j||+\eta||z_{j-1}-z^{\prime}_{j-1}||+(\rho_2-\rho_1)LZ(1-\beta)^j.\\
\end{aligned}
\end{displaymath}
Let $A_j = ||z_{j}-z^{\prime}_j||$, $b=D(\rho_1)+\eta$, $c=\eta$, $d=(\rho_2-\rho_1)LZ$ and $\alpha=1-\beta$, we have the following recurrence relation:
\begin{displaymath}
A_j = bA_{j-1}+cA_{j-2}+d\alpha^j.
\end{displaymath}
If $j\geq2$, consider $X_j=A_j+k\alpha^j$ subject to
\begin{displaymath}
X_j = bX_{j-1}+cX_{j-2}.
\end{displaymath}
Then
\begin{displaymath}
\centering
\begin{aligned}
A_j+k\alpha^j&=b[A_{j-1}+k\alpha^{j-1}]+c[A_{j-2}+k\alpha^{j-2}],\\
(d+k)\alpha^j&=bk\alpha^{j-1}+ck\alpha^{j-2},\\
k&=\frac{d\alpha^2}{b\alpha+c-\alpha^2}.\\
\end{aligned}
\end{displaymath}
Therefore by Lemma 2,
\begin{displaymath}
X_j=c_1r_1^j+c_2r_2^j,
\end{displaymath}
where $r_1$ and $r_2$ are the roots of characteristic equation of second order homogeneous linear recurrence relation $A_n = [D(\rho_1)+\eta]A_{n-1}+\eta A_{n-2}$. Now consider the coefficients $c_1$ and $c_2$.\\


We know that $A_0=||z_{0}-z_{0}||=0$, $A_1=||z_{1}-z^{\prime}_{1}||\leq (\rho_2-\rho_1)\frac{D(\rho_1)LZ}{\beta}$, and
\begin{displaymath}
\centering
\left \{
\begin{aligned}
X_0 &= A_0+k \\
X_1 &= A_1+k\alpha. 
\end{aligned}
\right.
\end{displaymath} 
Therefore, plugging $b=D(\rho_1)+\eta$, $c=\eta$, $d=(\rho_2-\rho_1)LZ$, $\alpha=1-\beta$, $k=\frac{d\alpha^2}{b\alpha+c-\alpha^2}$, $A_0=0$ and $A_1\leq (\rho_2-\rho_1)\frac{D(\rho_1)LZ}{\beta}$, we have
\begin{displaymath}
\begin{aligned}
X_0 &=  (\rho_2-\rho_1)\frac{LZ(1-\beta)^2}{[D(\rho_1)+\eta](1-\beta)+\eta-(1-\beta)^2}\\
&=R_0(\rho_2-\rho_1),\\
X_1 &\leq  (\rho_2-\rho_1)\frac{D(\rho_1)LZ}{\beta}+\frac{(\rho_2-\rho_1)LZ(1-\beta)^3}{[D(\rho_1)+\eta](1-\beta)+\eta-(1-\beta)^2}\\
& = R_1(\rho_2-\rho_1).
\end{aligned}
\end{displaymath}
Finally consider $c_1$ and $c_2$. We have
\begin{displaymath}
\centering
\left \{
\begin{aligned}
R_0(\rho_2-\rho_1) &= c_1+c_2 \\
R_1(\rho_2-\rho_1) &= c_1r_1+c_2r_2 .\\
\end{aligned}
\right.
\end{displaymath}
We solve the equations above and get
\begin{displaymath}
\centering
\left \{
\begin{aligned}
c_1 &= \frac{R_0r_2-R_1}{r_1-r_2}(\rho_2-\rho_1) \\
c_2 &= \frac{R_0r_1-R_1}{r_2-r_1}(\rho_2-\rho_1). \\
\end{aligned}
\right.
\end{displaymath}
Combining all above conclusions, we have
\begin{displaymath}
\begin{aligned}
A_j =& X_j-k\alpha^j\\
\leq & (\rho_2-\rho_1)[(\frac{R_0r_2-R_1}{r_2-r_1}r_1^j+\frac{R_0r_1-R_1}{r_1-r_2}r_2^j)-\frac{LZ(1-\beta)^{j+2}}{[D(\rho_1)+\eta](1-\beta)+\eta-(1-\beta)^2}]\\
 =& (\rho_2-\rho_1)G(\rho_1,\eta,j)\\
\end{aligned}
\end{displaymath}
where $\beta$, $L$ and $Z$ are constants.
\end{proof}

\section{Full Version and Proof of Lemma 5}
\setcounter{lemma}{4}
\begin{lemma}[Iteration error estimation fixing the step size $\rho$]
For any initial point $z_0\in \mathbb{R}^N$, iteration step $j \geq 2$, step sizes $\rho \in \mathbb{R}$, $\eta_1\in \mathbb{R}$, $\eta_2\in \mathbb{R}$ such that $\eta_1 \leq \eta_2$, we have  
\begin{displaymath}
\begin{aligned}
&||g^j(z_1,z_0,\rho,\eta_1)-g^j(z^{\prime}_1,z_0,\rho,\eta_2)||\\ 
\leq& (\eta_2-\eta_1)\bigg[(\frac{R_0r_2-R_1}{r_2-r_1}r_1^j+\frac{R_0r_1-R_1}{r_1-r_2}r_2^j)\\
&\quad -\frac{F(\eta^*)}{D(\rho)+2\eta_1-1}\bigg]\\
:=& (\eta_2-\eta_1)H(\eta^*,\rho,j),
\end{aligned}
\end{displaymath}
where $r_1$ and $r_2$ are the roots of the characteristic equation of second order homogeneous linear recurrence relation $A_n = [D(\rho)+\eta_1]A_{n-1}+\eta_1 A_{n-2}$,  
\begin{displaymath}
\begin{aligned}
R_0 = & \frac{F(\eta^*)}{D(\rho)+2\eta_1-1},\\
R_1 = & F(\eta^*)\frac{1}{\eta_1}\left[1+\frac{1}{D(\rho)+2\eta_1-1}\right],\\
F(\eta^*)=& \mathop{\rm{max}}_{
        n \in\{0,1,\cdots,j-1\}
        \atop
        \eta \in [\eta_l,\eta_u]} \eta^{n}\left[\rho LZ+\frac{\rho Z}{\eta-L}L\right]-\frac{\rho Z}{\eta-L}L^{n+1}.
\end{aligned}
\end{displaymath}
\end{lemma}

\begin{proof} By Lemma 3, 
\begin{displaymath}
\begin{aligned}
&||g^j(z_1,z_0,\rho,\eta_1)-g^j(z^{\prime}_1,z_0,\rho,\eta_2)||\\
=&||g(z_j,z_{j-1},\rho,\eta_1)-g(z^{\prime}_j,z^{\prime}_{j-1},\rho,\eta_2)||\\
=&||[z_{j}-\rho \nabla f(z_{j})-\eta_1(z_{j}-z_{j-1})]-[z^{\prime}_j-\rho \nabla f(z^{\prime}_j)-\eta_2(z^{\prime}_j-z^{\prime}_{j-1})]||\\
=&||[z_{j}-\rho \nabla f(z_{j})-\eta_1(z_{j}-z_{j-1})]-[z^{\prime}_j-\rho \nabla f(z^{\prime}_j)-\eta_1(z^{\prime}_j-z^{\prime}_{j-1})]+[(\eta_2-\eta_1)(z^{\prime}_j-z^{\prime}_{j-1})||\\
\leq& [D(\rho)+\eta_1]||z_{j}-z^{\prime}_j||+\eta_1||z_{j-1}-z^{\prime}_{j-1}||+(\eta_2-\eta_1)||z^{\prime}_j-z^{\prime}_{j-1}||.\\
\end{aligned}
\end{displaymath}
We first consider $||z^{\prime}_j-z^{\prime}_{j-1}||$. By Assumption 2,\\
\begin{displaymath}
\begin{aligned}
||z^{\prime}_j-z^{\prime}_{j-1}|| &= ||-\rho \nabla f(z^{\prime}_{j-1})-\eta_2(z^{\prime}_{j-1}-z^{\prime}_{j-2})|| \\ 
&\leq \rho ||\nabla f(z^{\prime}_{j-1})||+\eta_2||(z^{\prime}_{j-1}-z^{\prime}_{j-2})||\\
&\leq \rho L^{j-1}Z +\eta_2||(z^{\prime}_{j-1}-z^{\prime}_{j-2})||.\\
\end{aligned}
\end{displaymath}
Let $A_n = ||z_{n+1}-z^{\prime}_{n}||$, $a=\eta_2$, $b=\rho Z$ and $c=L$, we have the following recurrence relation:
\begin{displaymath}
A_n=aA_{n-1}+bc^n. \\
\end{displaymath}
If $n\geq1$, consider $X_j=A_j+kc^n$ subject to
\begin{displaymath}
X_n=aX_{n-1},
\end{displaymath}
we have 
\begin{displaymath}
\centering
\begin{aligned}
A_n+kc^n&=a(A_{n-1}+kc^{n-1}),\\
(b+k)c^n&=akc^{n-1},\\
k&=\frac{bc}{a-c}.\\
\end{aligned}
\end{displaymath}
Note that $\{X_n\}$ is geometric series, we have
\begin{displaymath}
X_n=a^nX_{0}.
\end{displaymath}
Now consider $X_0$. Given that $A_0=||z^{\prime}_{1}-z_{0}||=\rho LZ$, and we also have \\
\begin{displaymath}
\centering
\begin{aligned}
X_1&=A_1+kc\\
&=(bc+aA_0)+kc,\\
X_0&=\frac{1}{a}X_1\\
&=A_0+\frac{(b+k)c}{a}\\
&=\rho LZ+\frac{(\rho Z+\frac{\rho LZ}{\eta_2-L})L}{\eta_2}.
\end{aligned}
\end{displaymath}
Therefore 
\begin{displaymath}
\begin{aligned}
A_n =& X_n-kc^n\\
= & \eta_2^{n}[\rho LZ+\frac{(\rho Z+\frac{\rho LZ}{\eta_2-L})L}{\eta_2}]-\frac{\rho LZ}{\eta_2-L}L^{n}\\
=& \eta_2^{n}[\rho LZ+\frac{\rho Z(\eta_2-L)+\rho LZ}{\eta_2(\eta_2-L)}L]-\frac{\rho Z}{\eta_2-L}L^{n+1}\\
=& \eta_2^{n}[\rho LZ+\frac{\rho Z}{\eta_2-L}L]-\frac{\rho Z}{\eta_2-L}L^{n+1}.\\
\end{aligned}
\end{displaymath}
Then let $B_j = ||z_{j}-z^{\prime}_j||$, $b^{\prime} =D(\rho)+\eta_1$, $c^{\prime} =\eta_1$ and $d=\eta_2-\eta_1$, we have the following recurrence relation:\\
\begin{displaymath}
B_j \leq b^{\prime} B_{j-1}+c^{\prime}B_{j-2}+dA_{j-1}.
\end{displaymath}
For subsequent conclusions, it is beneficial to find a upper bound for $A_{j-1}$. Since $j$ is finite, there must exists a $n^* \in(0,1,\cdots,j-1)$, such that $A_{n^*}$ reaches the maximum value of series $A_n$, where $n \in(0,1,\cdots,j-1)$. Then let
\begin{displaymath}
\begin{aligned}
F(\eta^*)&= \mathop{\rm{max}}_{
        \eta \in [\eta_l,\eta_u]} A_{n^*}(\eta)
        \\&=\mathop{\rm{max}}_{
        n \in(0,1,\cdots,j-1)
        \atop
        \eta \in [\eta_l,\eta_u]} \eta^{n}[\rho LZ+\frac{\rho Z}{\eta-L}L]-\frac{\rho Z}{\eta-L}L^{n+1},
\end{aligned}
\end{displaymath}
we can see that at each iteration step, $A_{j}\leq A_{n^*} \leq F(\eta^*)$. Therefore, 
\begin{displaymath}
B_j \leq b^{\prime} B_{j-1}+c^{\prime}B_{j-2}+dF(\eta^*) := b^{\prime} B_{j-1}+c^{\prime}B_{j-2}+d^{\prime}
\end{displaymath}
by letting $d^{\prime}=dF(\eta^*)$. If $j\geq2$, consider $X_j=B_j+k$ subject to
\begin{displaymath}
X_j=b^{\prime}X_{j-1}+c^{\prime}X_{j-2},
\end{displaymath}
we have 
\begin{displaymath}
\centering
\begin{aligned}
B_j+k&=b^{\prime}[B_{j-1}+k]+c[B_{j-2}+k],\\
d^{\prime}+k&=b^{\prime}k+c^{\prime}k,\\
k&=\frac{d^{\prime}}{b^{\prime}+c^{\prime}-1}.\\
\end{aligned}
\end{displaymath}
Therefore by Lemma 2,
\begin{displaymath}
X_j=c_1r_1^j+c_2r_2^j,
\end{displaymath}
where $r_1$ and $r_2$ are the roots of characteristic equation of second order homogeneous linear recurrence relation $B_n = [D(\rho)+\eta_1]B_{n-1}+\eta_1 B_{n-2}$. Now consider the coefficients $c_1$ and $c_2$. Then consider the initial values $X_1$ and $X_0$.
Given that $B_0=||z_{0}-z_{0}||=0$, $B_1=||z_{1}-z^{\prime}_{1}||= ||[z_{0}-\rho \nabla f(z_0)]-[z_{0}-\rho \nabla f(z_0)]||$, and we also have
\begin{displaymath}
\centering
\left \{
\begin{aligned}
B_2 &= b^{\prime}B_1+c^{\prime}A_0+d^{\prime} \\
B_3 &= b^{\prime}B_2+c^{\prime}A_1+d^{\prime}, \\
\end{aligned}
\right.
\end{displaymath}
we have
\begin{displaymath}
\centering
\left \{
\begin{aligned}
B_2+k &= b^{\prime}X_1+c^{\prime}X_0 \\
B_3+k &= b^{\prime}X_2+c^{\prime}X_1. \\
\end{aligned}
\right.
\end{displaymath}
Solving the equation above, we have
\begin{displaymath}
\centering
\left \{
\begin{aligned}
X_1 &= \frac{1}{c^{\prime}}[d^{\prime}-(b^{\prime}-1)k] \\
X_0 &= \frac{1}{c^{\prime}}[(d^{\prime}+k)(1-\frac{b^{\prime}}{c^{\prime}})+\frac{b^{'2}}{c^{\prime}}k]. \\
\end{aligned}
\right.
\end{displaymath}
Plugging $b^{\prime}=D(\rho)+\eta_1$, $c^{\prime}=\eta_1$, $d^{\prime}=(\eta_2-\eta_1)F(\eta^*)$ and $k=\frac{d^{\prime}}{b^{\prime}+c^{\prime}-1}$, we have
\begin{displaymath}
\begin{aligned}
X_0 &=  \frac{1}{\eta_1}[[(\eta_2-\eta_1)F(\eta^*)+\frac{(\eta_2-\eta_1)F(\eta^*)}{D(\rho)+2\eta_1-1}](1-\frac{D(\rho)+\eta_1}{\eta_1})+\frac{[D(\rho)+\eta_1]^2}{\eta_1}\frac{(\eta_2-\eta_1)F(\eta^*)}{D(\rho)+2\eta_1-1}]\\
&=(\eta_2-\eta_1)\frac{F(\eta^*)}{\eta_1}[-(1+\frac{1}{D(\rho)+2\eta_1-1})\frac{D(\rho)}{\eta_1}+\frac{[D(\rho)+\eta_1]^2}{\eta_1(D(\rho)+2\eta_1-1)}]\\
&=(\eta_2-\eta_1)\frac{F(\eta^*)}{\eta_1}[\frac{D(\rho)(D(\rho)+2\eta_1)}{\eta_1(D(\rho)+2\eta_1-1)}+\frac{[D(\rho)+\eta_1]^2}{\eta_1(D(\rho)+2\eta_1-1)}]\\
&=(\eta_2-\eta_1)\frac{F(\eta^*)}{\eta_1}[\frac{\eta_1^2}{\eta_1(D(\rho)+2\eta_1-1)}]\\
&=(\eta_2-\eta_1)\frac{F(\eta^*)}{D(\rho)+2\eta_1-1}\\
&=R_0(\eta_2-\eta_1),\\
X_1 &\leq  \frac{1}{\eta_1}[(\eta_2-\eta_1)F(\eta^*)-[D(\rho)+\eta_1-1]\frac{(\eta_2-\eta_1)F(\eta^*)}{D(\rho)+2\eta_1-1}]\\
& = (\eta_2-\eta_1)\frac{F(\eta^*)}{\eta_1}[1-[D(\rho)+\eta_1-1]\frac{1}{D(\rho)+2\eta_1-1}]\\
& = (\eta_2-\eta_1)\frac{F(\eta^*)}{\eta_1}\frac{\eta_1}{D(\rho)+2\eta_1-1}\\
& = R_1(\eta_2-\eta_1).
\end{aligned}
\end{displaymath}

Finally, we consider $c_1$ and $c_2$. Specifically, we have
\begin{displaymath}
\centering
\left \{
\begin{aligned}
R_0(\eta_2-\eta_1) &= c_1+c_2 \\
R_1(\eta_2-\eta_1) &= c_1r_1+c_2r_2. \\
\end{aligned}
\right.
\end{displaymath}
We solve the equations above to obtain
\begin{displaymath}
\centering
\left \{
\begin{aligned}
c_1 &= \frac{R_0r_2-R_1}{r_2-r_1}(\eta_2-\eta_1) \\
c_2 &= \frac{R_0r_1-R_1}{r_1-r_2}(\eta_2-\eta_1). \\
\end{aligned}
\right.
\end{displaymath}
Combining all the above equations and inequalities, we have
\begin{displaymath}
\begin{aligned}
A_j =& X_j-k\\
\leq & (\eta_2-\eta_1)[\frac{R_0r_2-R_1}{r_2-r_1}r_1^j+\frac{R_0r_1-R_1}{r_1-r_2}r_2^j)-\frac{F(\eta^*)}{D(\rho)+2\eta_1-1}]\\
 =& (\eta_2-\eta_1)H(\eta^*,\rho,j)\\
\end{aligned}
\end{displaymath}
where $L$ and $Z$ are constants.
\end{proof}

\section{Proof of Lemma 6}
\begin{proof} By Lemma 3, 
\begin{displaymath}
\begin{aligned}
&||g^j(z_1,z_0,\rho_1,\eta_1)-g^j(z^{\prime}_1,z_0,\rho_2,\eta_2)||\\
=&||g(z_j,z_{j-1},\rho_1,\eta_1)-g(z^{\prime}_j,z^{\prime}_{j-1},\rho_2,\eta_2)||\\
=&||[z_{j}-\rho_1 \nabla f(z_{j})-\eta_1(z_{j}-z_{j-1})]-[z^{\prime}_j-\rho_2 \nabla f(z^{\prime}_j)-\eta_2(z^{\prime}_j-z^{\prime}_{j-1})]||\\
=&||[z_{j}-\rho_1 \nabla f(z_{j})-\eta_1(z_{j}-z_{j-1})]-[z^{\prime}_j-\rho_1 \nabla f(z^{\prime}_j)-\eta_1(z^{\prime}_j-z^{\prime}_{j-1})]\\&+[(\rho_2-\rho_1)\nabla f(z^{\prime}_j)]+[(\eta_2-\eta_1)(z^{\prime}_j-z^{\prime}_{j-1})||\\
\leq &||[z_{j}-\rho_1 \nabla f(z_{j})-\eta_1(z_{j}-z_{j-1})]-[z^{\prime}_j-\rho_1 \nabla f(z^{\prime}_j)-\eta_1(z^{\prime}_j-z^{\prime}_{j-1})]||\\&+|\rho_2-\rho_1|||\nabla f(z^{\prime}_j)||+|\eta_2-\eta_1|||z^{\prime}_j-z^{\prime}_{j-1}||\\
\leq &[||[z_{j}-\rho_1 \nabla f(z_{j})-\eta_1(z_{j}-z_{j-1})]-[z^{\prime}_j-\rho_1 \nabla f(z^{\prime}_j)-\eta_1(z^{\prime}_j-z^{\prime}_{j-1})]||  + |\rho_2-\rho_1|||\nabla f(z^{\prime}_j)||]\\&+[||[z_{j}-\rho_1 \nabla f(z_{j})-\eta_1(z_{j}-z_{j-1})]-[z^{\prime}_j-\rho_1 \nabla f(z^{\prime}_j)  -\eta_1(z^{\prime}_j-z^{\prime}_{j-1})]||+|\eta_2-\eta_1|||z^{\prime}_j-z^{\prime}_{j-1}||]\\
\leq &|\rho_2-\rho_1|G(\rho_1,\eta_1,j)+|\eta_2-\eta_1|H(\eta^*,\rho_1,j)
\end{aligned}
\end{displaymath}
\end{proof}

\section{Proof of Theorem 7}
\begin{proof}
Let's assume that $M_{\rho_1,\eta_1}$, the number of iteration time for $A_{\rho_1,\eta_1}$ is larger than $M_{\rho_2,\eta_2} $, and denote $M_{\rho_1,\eta_1}$ by $M$. By Lemma 6, $||g^j(z_1,z_0,\rho_1,\eta_1)-g^j(z^{\prime}_1,z_0,\rho_2,\eta_2)|| \leq  |\rho_2-\rho_1|G(\rho_1,\eta_1,j)+|\eta_2-\eta_1|H(\eta^*,\rho_1,j)$, therefore,\\
\begin{displaymath}
\begin{aligned}
&||c(A_{\rho_1,\eta_1} ,x)-c(A_{\rho_2,\eta_2} ,x)| 
\\
=& |\sum_{j=0}^{M}||z^{*}-g^j(z_1,z_0,\rho_1,\eta_1)||-\sum_{j=0}^{M}||z^{*}-g^j(z^{\prime}_1,z_0,\rho_2,\eta_2)|||\\
\leq&\sum_{j=0}^{M}[||g^j(z_1,z_0,\rho_1,\eta_1)||-||g^j(z^{\prime}_1,z_0,\rho_2,\eta_2)||]\\
\leq& \sum_{j=0}^{M}[||g^j(z_1,z_0,\rho_1,\eta_1)-g^j(z^{\prime}_1,z_0,\rho_2,\eta_2)||]\\
=& ||z_1-z^{\prime}_1||+\sum_{j=1}^{M}[||g^j(z_1,z_0,\rho_1,\eta_1)-g^j(z^{\prime}_1,z_0,\rho_2,\eta_2)||]\\
\leq& |\rho_2-\rho_1|\frac{LZD(\rho_1)}{\beta}+\sum_{j=1}^{M}[|\rho_2-\rho_1|G(\rho_1,\eta_1,j)+|\eta_2-\eta_1|H(\eta^*,\rho_1,j)]\\
=& |\rho_2-\rho_1|\frac{LZD(\rho_1)}{\beta}+|\rho_2-\rho_1|\sum_{j=1}^{M}G(\rho_1,\eta_1,j)+|\eta_2-\eta_1|\sum_{j=1}^{M}H(\eta^*,\rho_1,j).
\end{aligned}
\end{displaymath}
Recall that $M \leq H$, plugging in the bound in Assumption 2 gives the desired result.
\end{proof}
\section{Proof of Lemma 7}
\begin{proof}
By Lemma 4, for any $j\geq 2$,
\begin{displaymath}
\begin{aligned}
&G(\rho_1,\eta_1,j) = (\frac{R_0r_2-R_1}{r_2-r_1}r_1^j+\frac{R_0r_1-R_1}{r_1-r_2}r_2^j)
-\frac{LZ(1-\beta)^{j+2}}{[D(\rho_1)+\eta_1](1-\beta)+\eta_1-(1-\beta)^2}.
\end{aligned}
\end{displaymath}
Note that $\rho_1$ and $\eta_1$ are constants, $G(\rho_1,\eta_1,j)$ can be rewritten as
\begin{displaymath}
G(\rho_1,\eta_1,j)=c_1r_1^j+c_2r_2^j-\frac{c_3^{j+2}}{c_4},
\end{displaymath}
where $c_1 = \frac{R_0r_2-R_1}{r_2-r_1}$, $c_2 = \frac{R_0r_1-R_1}{r_1-r_2}$, $c_3 = LZ(1-\beta)$ and $c_4 = [D(\rho_1)+\eta_1](1-\beta)+\eta_1-(1-\beta)^2$.\\
We first consider $c_3$ and $c_4$. By definition, $\beta \in (0,1)$, it is easily to see that $c_3>0$ and $c_4>0$. Therefore
\begin{displaymath}
G(\rho_1,\eta_1,j)<c_1r_1^j+c_2r_2^j.
\end{displaymath}
Now consider $c_1$. The denominator $r_1-r_2<0$, for the numerator, note that
\begin{displaymath}
\begin{aligned}
R_0 = &\frac{LZ(1-\beta)^2}{[D(\rho_1)+\eta_1](1-\beta)+\eta_1-(1-\beta)^2}\\
R_1 = &\frac{D(\rho_1)LZ}{\beta}+\frac{LZ(1-\beta)^3}{[D(\rho_1)+\eta_1](1-\beta)+\eta_1-(1-\beta)^2}\\
\end{aligned}
\end{displaymath}
and
\begin{displaymath}
\begin{aligned}
&[D(\rho_1)+\eta_1](1-\beta)+\eta_1-(1-\beta)^2 \\\geq & [D(\rho_1)+\eta_1](1-\beta)-(1-\beta)^2\\
= &(1-\beta)[D(\rho_1)+\eta_1-(1-\beta)]\\
\geq & 0,\\
\end{aligned}
\end{displaymath}
where the last line follows from $0<1-\beta<1$, $D(\rho_1)>1$ and $\eta_1>0$. Therefore $R_0>0$ and $R_1>0$, and the numerator of $c_1$, $R_0r_2-R_1<0$. Consequently, $c_1>0$.
Then consider $r_1$ and $r_2$. By Vieta's theorem
\begin{displaymath}
\begin{aligned}
&r_1+r_2 = D(\rho_1)+\eta_1,\\
&r_1r_2 = -\eta_1.\\
\end{aligned}
\end{displaymath}
By the definition of $D(\rho)$, $r_1>0$, $r_2<0$ and $r_1+r_2>1$. Therefore $r_1>1$ and $r_1>|r_2|$.\\
Therefore
\begin{displaymath}
\begin{aligned}
G(\rho_1,\eta_1,j)< &c_1r_1^j+c_2r_2^j\\
< &c_1r_1^j+|c_2r_2^j|\\
= &c_1r_1^j+|c_2||r_2|^j\\
< &c_1r_1^j+|c_2||r_1|^j\\
=&(c_1+|c_2|)r_1^j\\
=&R^*r_1^j,\\
\end{aligned}
\end{displaymath}
and
\begin{displaymath}
\begin{aligned}
\sum_{j=1}^{M}[G(\rho_1,\eta_1,j)] < & \sum_{j=1}^{M}R^*r_1^j\\
< &\sum_{j=1}^{H}R^*r_1^j\\
= &R^*[r_1\frac{1-r_1^H}{1-r_1}],
\end{aligned}
\end{displaymath}
\begin{displaymath}
\begin{aligned}
& \frac{LZD(\rho_1)}{\beta}+\sum_{j=1}^{M}[G(\rho_1,\eta_1,j)
< \frac{LZD(\rho_1)}{\beta}+R^*[r_1\frac{1-r_1^H}{1-r_1}]
= & G^*
\end{aligned}
\end{displaymath}
\end{proof}
\section{Proof of Lemma 8}
\begin{proof}
By Lemma 5, for any $j$,
\begin{displaymath}
H(\eta^*,\rho_1,j) = (\frac{R_0r_2-R_1}{r_2-r_1}r_1^j+\frac{R_0r_1-R_1}{r_1-r_2}r_2^j)-\frac{F(\eta^*)}{D(\rho_1)+2\eta_1-1}.
\end{displaymath}
Note that $\rho_1$, $\eta_1$ and $F(\eta^*)$ are constants, $H(\eta^*,\rho_1,j)$ can be rewritten as
\begin{displaymath}
H(\eta^*,\rho_1,j)=c_1r_1^j+c_2r_2^j-c_3,
\end{displaymath}
where $c_1 = \frac{R_0r_2-R_1}{r_2-r_1}$, $c_2 = \frac{R_0r_1-R_1}{r_1-r_2}$, $c_3 =\frac{F(\eta^*)}{D(\rho_1)+2\eta_1-1}$.\\
We first consider $c_3$. Note that
\begin{displaymath}
\begin{aligned}
F(\eta^*)=& \mathop{\rm{max}}_{
        n \in(0,1,\cdots,j-1)
        \atop
        \eta \in [\eta_l,\eta_u]} \eta^{n}[\rho_1 LZ+\frac{\rho_1 Z}{\eta-L}L]-\frac{\rho_1 Z}{\eta-L}L^{n+1}\\
\geq & [\rho_1 LZ+\frac{\rho_1 Z}{\eta_1-L}L]-\frac{\rho_1 Z}{\eta_1-L}L\\
= & \rho_1 LZ\\
> & 0,\\
\end{aligned}
\end{displaymath}
and by definition, $D(\rho_1)+2\eta_1-1>0$, $c_3>0$. Therefore
\begin{displaymath}
H(\eta^*,\rho_1,j)<c_1r_1^j+c_2r_2^j.
\end{displaymath}
Now consider $c_1$. The denominator $r_1-r_2<0$, for the numerator, note that
\begin{displaymath}
\begin{aligned}
R_0 = & \frac{F(\eta^*)}{D(\rho_1)+2\eta_1-1}\\
R_1 = & F(\eta^*)\frac{1}{\eta_1}[1+\frac{1}{D(\rho_1)+2\eta_1-1}],\\
\end{aligned}
\end{displaymath}
and due to $D(\rho_1)>1$ and $\eta_1>0$,
\begin{displaymath}
D(\rho_1)+2\eta_1-1>0.
\end{displaymath}
Therefore $R_0>0$ and $R_1>0$, and the numerator of $c_1$, $R_0r_2-R_1<0$ . Consequently, $c_1>0$.
Then consider $r_1$ and $r_2$. By Vieta's theorem, we have
\begin{displaymath}
\begin{aligned}
&r_1+r_2 = D(\rho_1)+\eta_1,\\
&r_1r_2 = -\eta_1.\\
\end{aligned}
\end{displaymath}
By the definition of $D(\rho)$, $r_1>0$, $r_2<0$ and $r_1+r_2>1$. Therefore $r_1>1$ and $r_1>|r_2|$. 
Therefore
\begin{displaymath}
\begin{aligned}
H(\eta^*,\rho_1,j)< &c_1r_1^j+c_2r_2^j\\
< &c_1r_1^j+|c_2r_2^j|\\
= &c_1r_1^j+|c_2||r_2|^j\\
< &c_1r_1^j+|c_2||r_1|^j\\
=&(c_1+|c_2|)r_1^j\\
=&R^{*'}r_1^j,\\
\end{aligned}
\end{displaymath}
and \\
\begin{displaymath}
\begin{aligned}
\sum_{j=1}^{M}[H(\eta^*,\rho_1,j)] < & \sum_{j=1}^{M}[R^{*'}r_1^j]\\
< &\sum_{j=1}^{H}[R^{*'}r_1^j]\\
= &R^{*'}[r_1\frac{1-r_1^H}{1-r_1}]\\
= &H^*.
\end{aligned}
\end{displaymath}
\end{proof}
\section{Proof of Theorem 9}
\begin{proof}
The pseudo-dimension of
\begin{displaymath}
\mathcal{A}_{N_\rho,N_\eta}=\{\mathit{A}_{\rho,\eta}, \rho \in N_\rho, \eta \in N_\eta\} 
\end{displaymath}
is at most $[\log(\nu/LZ)/\log(1-\beta)]\log(|N_\rho|)\log(|N_\eta|)$. Moreover, since $\mathcal{A}_{N_\rho,N_\eta}$ is finite and $\log|N_{\rho}|\sim \tilde{O}(H)$ and $\log|N_{\eta}|\sim \tilde{O}(H)$, it also trivially admits an ERM algorithm $L_{N_\rho,N_\eta}$, which $(\epsilon,\delta)$-learns the optimal algorithm in $\mathcal{A}_{N_\rho,N_\eta}$ using $m=\tilde{O}(H^2\log(|N_\rho|)\log(|N_\eta|)/\epsilon^2)$ samples.
\end{proof}


\end{document}